\documentclass[12pt]{article}
\usepackage{indentfirst,amsthm,amssymb,url}
\newtheorem{theorem}{Theorem}
\newtheorem{definition}{Definition}
\newtheorem{proposition}{Proposition}
\newtheorem{corollary}{Corollary}
\newtheorem{lemma}{Lemma}
\newtheorem{remark}{Remark}
\newtheorem{algorithm}{Algorithm}

\title{The Generalized Locker Problem }
\author{Keneth Adrian P. Dagal\\
\emph{kendee2012@gmail.com}\\
Department of Mathematics\\
Far Eastern University\\
Manila, Philippines}
\date{}
\begin{document}

\maketitle

\begin{abstract}
The Locker Problem is frequently used in introducing some topics in elementary number theory like divisors and multiples. It appears in many curricula ranging from elementary, secondary and up to tertiary level. In this paper, I will provide the structure of the problem and algorithms in solving some modified problems.
\end{abstract}		

\section{Introduction}
\noindent The original locker problem stated:

\emph{There are 100 students and 100 lockers in a certain school, each student has a certain locker number. The students and the lockers are uniquely numbered from 1 to 100. The students had a certain game. Initially, the lockers are all closed. The first student opened all the lockers. The second student closed every second locker.The third student opened or closed every third locker.This method goes until the 100th student. After all the students finish opening and closing the lockers, how many lockers are left opened ? and what are the lockers which are left opened ?
}

The locker problem varies in terms of the number of lockers and students. But the number of lockers and students is assumed to be equal.In most cases, the number is 100, but there are cases wherein the number is 20 or 1000. The answer is known that those lockers whose number is a perfect square will be left opened.

In this paper, we consider the possibility that the number of students and the number of lockers is unequal. We also consider the possibility of students repeating turns and the possibility that some students will not participate in the activity.But we preserve the rule of the problem which is every $i$th student will change the state of all lockers numbered $j$ where $i\mid j$. In addition to this, I will provide algorithms in finding the corresponding open or close lockers given a subset of students and vice versa.

\section{Some Observations}

We have to note that: \emph{Every locker number $j$ is assigned to a particular student $i$ where $i=j$}.

\begin{lemma}\label{thm:le1}
Let $A$ be the set containing all students, $B$ be the set containing all the lockers. Then, the number of elements of $A$ is equal to the number of elements of $B$, $|A| =|B|$.
\end{lemma}
\begin{proof}
Suppose $|A| \neq |B|$. If $|A| < |B|$, then there exists lockers that are not assigned to students or assigned but to students who already had assignments, by pigeonhole, which contradicts our assumption. Now, if $|A| > |B|$ , then there exists students that have no assigned locker number whose numbers are greater than the largest locker number. But in the problem rule, these students will not have any effect on the outcome of the state of the lockers since every student numbered $k+a$, $a, k \in \mathbb{N}$,where $k$ is the largest locker number and $a$ be the difference of the student number and  the largest locker number,does not have a multiple from $1$ to $k$ since the least possible multiple of $k+a$ is $k+a$ which is greater than $k$. Thus, $|A| =|B|$.
\end{proof}

The  lemma \ref{thm:le1} shows that for every locker uniquely assigned to every student,any additional student is immaterial.
\begin{lemma}
Given $A$ and $B$.
\begin{enumerate}
	\item Let $S$ be the set of $r$-element subset of $A$ where $r = 0,1,2,..., n$ and $ |A| = n $. Then $ |S| = 2^n. $
	\item Let $L$ be the set of binary strings of 0 and 1 denoted by $\{s_j\}^n$ such that every $j$th position of the string denote the state of the $j$th locker defined as $s_j=\left\{\begin{array}{ccc}0 & \mathrm{if} & j \textnormal{th locker is close} \\1 & \mathrm{if} & j \textnormal{th locker is open} \end{array}\right.$ Then $ |L| = 2^n. $
\end{enumerate}
\end{lemma}

\begin{proof}
The first claim is straightforward since $\sum_{i=0}^n{n\choose i} = 2^n$. The second claim is also straightforward since $\underbrace{2\cdot2\cdots2\cdot2\cdot2}_{n}=2^n.$
\end{proof}

Every $u \in S$ has a corresponding $v \in L$.We will show that the correspondence is \emph{one-to-one} and that follows bijection since we know that $|S| =|L|$. We define $f: S \rightarrow L $ where every subset of $A$ has a corresponding result of states of the lockers. As an example, if $a_i$ is the only student who changes the state of the lockers, then the corresponding string will be all 0's except those positions which are multiple of  $i$ considering that the initial string is consists of all 0's. For now, we consider initial states of the locker to be all closed which implies all $s_j$'s are 0's. But later on, we generalize it to any random combination of state of lockers.

\begin{theorem}
Suppose $u_1$ is a 1-element subset of $A$ and $u_1 \in S$. Then there is no other $u \in S$ such that $u_1$ and $u$ have the same image $v \in L$.
\end{theorem}

\begin{proof}
Considering the initial string consisting of all 0's. If $u_1 = \{ a_i\}$ , then all $s_j$'s of string $v$ are $1$ where $i\mid j$. This implies that all $j$'s less than $i$ are 0's. We look for another $u \in S$ such that $u \neq u_1$ and the image of $u$ is also $v$.  If the image of $u$ is $v$, then all $j$'s less than $i$ are 0's. For this, there is no $a_k \in A$ for all $k < i $ where $a_k \in u$.  Suppose $a_1 \in u$, then $s_1 = 1 $ which contradicts our assumption that $s_1 = 0$. Inductively, suppose $a_k \in u$. Since every $a_j \notin u$ where $j = 1,2, 3 , ..., k-1$, then that assures that $s_j = 1$ which contradicts our assumption that $s_j = 0$. Now since $u \neq u_1$ and it follows that $u_1 \subset u$, we consider every $s_j$, where $ ci < j < (c+1)i\,\, \forall\,\, c = 1,2,3,...,\left\lfloor \frac{n}{i}\right\rfloor-1$.All $s_j = 0$. Every $j$th locker will only be touched once by every $a_j$ student, then it follows that $a_j \notin u$. We only left for $j = (c+1)i$.Similar reasoning as above, we show that  $u = u_1$.
\end{proof}

\begin{corollary}
Every $u \in S$ is uniquely assigned to $v \in L$.
\end{corollary}
\begin{proof}
Now suppose there are two $u_1,u_2 \in S$ where the image of $u_1$ and $u_2$ is $v$. WLOG, we let $min\{a_i\} \in u_1$ such that $a_i \notin u_2$. But by Theorem 1, $a_i \in u_2$. 
\end{proof}

By the lemma 2 and corollary 1, we have shown that $f$ \emph{is a bijective map} where $f$ is defined as for every $a_i \in u$ and $u \in S$ will change the state of all lockers numbered $j$ ( the $j$th position in $v \in L$) where $i\mid j$.

\section{The Structure of the Locker Problem}

We consider first the sequence of students who will change the state of the lockers. We consider repeating turns of students too.We have take into consideration the possibility that there are students who will not participate in the game. 

\begin{theorem}
Let $S$ be the set of $r$-element subset of $A$ where $r = 0,1,2,..., n$ and $ |A| = n $. "$\rightarrow$ " operation is defined to be "next to change the state of the locker". Then,  $\left\langle S, \rightarrow \right\rangle$ is an \emph{abelian group}.
\end{theorem}

\begin{proof}
Let $a_i$ and $a_j$ be distinct students. $a_i \rightarrow a_j = a_j \rightarrow a_i $. Let a particular $b_j \in B$. The change for $b_j$ is the same whoever comes first. And thus, WLOG,we can let every $a_i \in u$ to be in ascending order.Similarly, associativity follows. Considering $v= \{0\}^n$ as the initial state of the lockers where the pre-image of $v$ is $\emptyset$, then for every $u_i \in S$, $u_i \rightarrow \emptyset = \emptyset \rightarrow u_i$. It is easy to see that $ u_i \rightarrow u_i =(u_i)^2 = \emptyset$. Since every student $a_i$ changes the state of the locker twice, that means for a particular locker $b_j$, if it is not touched by $a_i$, it is clear. But if it is touched twice,then 0 becomes 1, 1 becomes 0 or 1 becomes 0, and 0 becomes 1. And lastly, for every $u_1, u_2 \in S$, $u_1 \rightarrow u_2 \in S$. This is clear since $S$ is a power set of $A$.
\end{proof}

\begin{remark}
Suppose $u_1 \cap u_2 = \emptyset$, then clearly $u_1 \rightarrow u_2 \in S$. Now,suppose $u_1 \cap u_2 \neq \emptyset$, then $u_1 \rightarrow u_2 \in S$ since for all $ a_i \in u_1 \cup u_2 $ appears twice in the sequence which will make all these $a_i$'s not affect the state since $(\{a_i\})^2= \emptyset$.And thus, $u_1 \oplus u_2 \in S$. It should be noted that the structure of $\left\langle S, \rightarrow \right\rangle$ is the same as the structure of $\left\langle \wp (N), \oplus  \right\rangle$, where $\wp(N)$ is the power set of $N$ and $\oplus$ is the symmetric difference operation.
\end{remark}

It can be seen that $L = \mathbb{Z}_2^n$ where $n$ means the number of $\mathbb{Z}_2$ in a row. $\left\langle L, +  \right\rangle$ is also an abelian group.The + is defined as position- wise addition in the string and knowing that  $\left\langle \mathbb{Z}_2, +  \right\rangle$ is an abelian group where + is defined to be addition modulo 2, everything is barely straightforward. Another immediate consequence of the theorem is that if a certain $a_i$ repeat his turn $k$ times, then if $k$ is even, $a_i$ does not affect the state of the lockers,otherwise he does, but same effect when he did not repeat his turn.

\begin{theorem}
 Given $f: S \rightarrow  L$. Then,  $f$ is an isomorphism map.
\end{theorem}

\begin{proof}
We know that $f$ is a bijective map. This implies that every $u \in S$ has a unique image $v \in L$ and every $v \in L$ is an image. If $ f(u) = v $ and $ f(u') = v' $, then $ f(u \rightarrow u') = f(u) + f(u') = v + v'$.This is true since for a particular locker $b_j$ with a state $s_j$, if for every $a_i \in u \oplus u'$ touch the locker, then that consists the number of $i$'s such that $i\mid j$. In fact, if the number of divisors in $u$ and $u'$ is of different parity, then  $s_j = 1$, otherwise $s_j=0$. With this,$ f(u \rightarrow u') = f(u) + f(u') = v + v'$. Thus, this proves that $f$ is an isomorphism map.
\end{proof}

We take into consideration $\emptyset \rightarrow \{0\}^n$, that is all lockers are initially closed as our initial state. Actually, suppose initially it is a random combination of open or close lockers, we will show that it can still be modified to the original consideration as $\emptyset \rightarrow \{0\}^n$ to make the problem simpler. 

\begin{theorem}
Suppose in general the initial $v\in L$ is not necessarily $v=\{0\}^n$. Then, $f = v + g$ and $f + v =g$, where $f:(\emptyset \rightarrow S) \rightarrow L$, $g:(u \rightarrow S) \rightarrow L$ for all $u \in S$, and $f(u)= v$.
\end{theorem}

\begin{proof}
 Since $f$ is an isomorphism map , then every $v \in L$ is an image of a unique $u \in S$ and also
\begin{eqnarray*}
v+g & = & f(u)+f(u \rightarrow S )\\
& = & f(u) +f(u) + f(S)\\
& = & f(S)\\
& = & f(\emptyset \rightarrow S)\\
& = & f. 
\end{eqnarray*}
Now since $f= v+g$ and $v+v = \{0\}^n$, then  $f +v= g+ v + v = g $.
\end{proof}

\begin{remark}
Take note that  $\left\langle L, +  \right\rangle$ is also an abelian group. Do not be confused the operation "$\rightarrow$" on S with the mapping notation "$\rightarrow$".
\end{remark}

\section{Some Modified Problems}

 For the remaining part, we now look for $u \in S$ such that we want to have $v \in L$ or vice versa. In particular,if we send all students, then the perfect square locker numbered will remain open.We first state some well-known theorems.

\begin{theorem}
The number of divisors of $n$, denoted by $d(n)$, is given by the formula $$d(n) = \prod_{i = 1 }^ k ( \alpha_i + 1 )$$ where  $n = \prod_{i = 1 }^ k {p_i}^{\alpha_i}$, $p_i $ is the $i$th prime number and $\alpha_i$  is the number of $p_i$ factors of $n$.
\end{theorem}

\begin{corollary}
If $n$ is a perfect square, then $d(n)$ is odd, otherwise $d(n)$ is even.
\end{corollary}

\begin{theorem}
Given $f: A \rightarrow v$. Then, $v = \{s_j\}^n $ where $s_j=\left\{\begin{array}{ccc}0 & \mathrm{if} & j \neq p^2 \\ 1 & \mathrm{if} & j=p^2 \end{array}\right.$ for some natural numbers $p$.
\end{theorem}
\begin{proof}
This is an immediate consequence of the Theorem 5 and Corollary 2.
\end{proof}

\begin{proposition}
Given $f: u \rightarrow v$ and $u$ is a 1-element subset of $A$.Then, $v = \{s_j\}^n $ where $s_j=\left\{\begin{array}{ccc}0 & \mathrm{if} & i\nmid j \\1 & \mathrm{if} & i \mid j \end{array}\right.$
\end{proposition}
\begin{proof}
By the rule of the game and considering $\{0\}^n$ as the initial state.
\end{proof}

\begin{theorem}
Given $f: u \rightarrow v$ and $u$ is a (n-1)-element subset of $A$. Then, $v = \{s_j\}^n $ where $s_j=\left\{\begin{array}{ccc}0 & \mathrm{if} & (j\neq p^2 \wedge i\nmid j) \vee (j = p^2 \wedge i\mid j)  \\1 & \mathrm{if} &  (j\neq p^2 \wedge i\mid j) \vee (j = p^2 \wedge i\nmid j) \end{array}\right.$ for some natural numbers $p$.
\end{theorem}
\begin{proof}
 $u = A - \{a_i\} = A \oplus \{a_i\} = A \rightarrow  \{a_i\}$.
\begin{eqnarray*}
v & = & f(u)\\
& = & f(A \rightarrow  \{a_i\})\\
& = & f(A) + f(\{a_i\}). 
\end{eqnarray*}
By Theorem 6 and Proposition 1, we are left with four cases: $ 0+0 =1+1= 0$ and $0+1 =1+0 = 1$.These cases are the  $s_j$'s for $f(A) , f(\{a_i\})$ respectively.
\end{proof}

The following definitions are redefined in this paper but taken in the paper of \cite{Torrence}. The refinement is done to make the notations suit the notations in previous Sections.

\begin{definition}
Let $n \in \mathbb{N}$ and $n = \prod_{i = 1 }^ k  {p_i}^{\alpha_i}$, $p_i $ is the $i$th prime number and $\alpha_i$  is the number of $p_i$ factors of $n$. The signature of $n$, denoted by $\varsigma (n)$, is the set of all $\alpha_i$'s in $n$. In other words, $$\varsigma(n) = \{ \alpha_i \mid i = 1,2,3,...k \}$$
\end{definition}

\begin{definition}
Let $T \subseteq \mathbb{N}$. $\sigma(T) = \{n \in \mathbb{N} \mid \varsigma(n) \subseteq T \}$
\end{definition}

Suppose $u = \sigma(T)$ where $T = mk , m = 2, 3, 4, ...$ for some $k \in \mathbb{N}$. It is noted that  if $\sigma(\{1\})$ is the set of squarefree numbers.

\begin{definition}
Let $T \subseteq \mathbb{N}$. $\upsilon(T) = \{n \in \mathbb{N} \mid n \equiv 0, 1, 2,..., (m-1) \,\,\textnormal{mod}\,\,\, (2m )\} $
\end{definition}

\begin{theorem}
Let $T \subseteq \mathbb{N}$. Given $f: u \rightarrow v$.If $u = \sigma(T)$, then $v = \{s_j\}^n$ where $s_j=\left\{\begin{array}{ccc}0 & \mathrm{if} & j \notin \sigma(\upsilon(T) ) \\1 & \mathrm{if} & j \in \sigma(\upsilon(T) )\end{array}\right.$ If $v=\sigma(T)$, then all $a_i \in u$ where $i \in \sigma(\upsilon^{-1}(T))$ must march.
\end{theorem}

\begin{proof}
 The proof is in the paper of \cite{Torrence}.
\end{proof}

For other results, see the paper of \cite{Torrence}.

Suppose we add another rule, that is $a_i$ can change the state of the lockers if $a_{i-1}$ changed the state of the lockers. In other words,if $a_i$ is allowed to participate, it is a must that all $a_k$ where $k = 1,2,3, .., i-1$ must change the lockers first. Our concern now is to count the number of open lockers, denoted by $\omega$. Clearly, $ n- \omega$ will be the number of closed lockers. It is also clear that for the original problem, the number of open lockers will be $ \left\lfloor \sqrt{n}\right\rfloor $. We denote $ \left\lfloor \sqrt{n}\right\rfloor  = \theta(n)$.

\begin{theorem}
Suppose  $u = \{ a_i \mid i = 1, 2, 3, ..., p \} $ where $p > \frac{n}{2}$. Then $\omega = \theta(p) + (( n - p) - (\theta(n) - \theta (p)))$.
\end{theorem}

\begin{proof}
For all $ j \leq p $, the number of 1's in $v$ up to $p$th position is $\theta(p).$ For all $j$, $ p < j \leq n $, the number of 1's in $v$ from $p$th position to $n$th position is $(( n - p) - (\theta(n) - \theta (p))).$ This is true since for every $s_j$, $j > p$ , if $s_j = 0$ becomes $s_j = 1$ and vice versa. Since $p > \frac{n}{2} \Rightarrow  2p > n$, this follows that every $a_i \in u' ( u'$ is the complement of $u )$, where $i > p $ will not touch every corresponding lockers from $p$ to $n$ once. Thus, $\omega = \theta(p) + (( n - p) - (\theta(n) - \theta (p)))$. 
\end{proof}

\section{Algorithms for All Modified Problems}

The previous section provides some straightforward solutions to some particular cases. In this section, we will have algorithms in looking for $u \in S$ such that we want to have $v \in L$ or vice versa. This may not be that efficient unlike the solutions provided in the previous section but this is more powerful than them since it solves any $u \in S$, we can have the corresponding $v \in L$ and vice versa.

Recall that $\sigma(\{1\})$ is the set of squarefree numbers. And we let  set $k \cdot \sigma(\{1\}) =\{ k\cdot s | s \in \sigma(\{1\})\}$. We also let $\xi (P)$ be the set of $k$-element subset of P, where P is the set of primes.Note that $|\sigma(\{1\})|= |\xi (P)|$. In fact, we can have an injective set map $h: \sigma(\{1\}) \rightarrow \xi (P)$ where for each $x \in \sigma(\{1\}) $,  $h(x) = \{p_i \mid x=  \prod_{i = 1 }^ k {p_i}\}$. One trivial example is $h(1) = \emptyset$. This note is very important in the next theorem since for all nonprime elements of $\sigma(\{1\})$, it can be treated as subsets in $\xi (P)$ which are not 1-element subset in $\xi (P)$. Furthermore, the 1-element subset in $\xi (P)$ are the prime elements of $\sigma(\{1\})$.In addition to this, we can partition $\xi (P)$ into $$\bigcup_{z=0}^\infty \phi(z)$$ where $\phi(z)=\{ h(R) \mid R \subset \sigma(\{1\})\wedge |R| = z\}$.

\begin{theorem}
$u = \{ a_i \mid i \in k \cdot \sigma(\{1\})\}$ if and only if  $v = \{s_j\}^n $ where $s_j= \left\{\begin{array}{ccc}0 & \mathrm{if} & j \neq k \\ 1 & \mathrm{if} & j = k \end{array}\right.$, or simply $s_k$.
\end{theorem}

\begin{proof}
WLOG, we can let $k=1$, this implies that $a_1 \in u$. All lockers touched by student $a_1$ will be closed except locker $b_1$, the first locker. This is true since every $i$ of $a_i \in u $ is $  \prod_{i = 1 }^ k {p_i}$. By Theorem $5$ and Corollary $2$, $d(i)$ is even which means closed lockers.In fact, $d(i)$ is a power of 2. Conversely, we let $v={s_k}$. Clearly, all  $a_x$ where $x < k $ for all $x\in \mathbb{N}$ are not elements of $u$ and $a_k \in u$. Now all lockers $b_j$ where $k \mid j$ will be opened. But we need to close all lockers $b_j$ where $k \mid j$ except locker $b_k$. Thus,  $a_{pk} \in u$ where $p \in P$ and $a_{p_1p_2k} \in u$ where $p_1,p_2 \in P$. Inductively, we have $ u = {a_k} \cup a_{pk} \cup \cdots =   \{ a_i \mid i \in k \cdot M \} $ where $$M = \bigcup_{z=0}^\infty \Phi(z)=\sigma(\{1\})$$  and where $ \Phi(z) = \{ h^{-1}(R) \mid R \subset \xi(P) \wedge |R| = z\}$. And thus, $u = \{ a_i \mid i \in k \cdot \sigma(\{1\}) \}$.    
\end{proof}

From now on, we consider the notation $s_k= f(u_k)$. Let us state the two main algorithms.

\begin{algorithm}
\textbf{\emph{Finding $v \in L$ given $u \in S$}}.\\
$u := n$\\ 
\textbf{while} $n \neq \emptyset$\\
\textbf{begin}\\ 
$ min(n) := a_k$\\
$ n := n \rightarrow u_k $\\
\textbf{print:} $k$\\
\textbf{end}
\end{algorithm}

\begin{proof}
Given $u \in S$ and let $u=u^1$. Find the $ min(u^1) = a_{k_1}$.Then, $u^1 = u_{k_1} \rightarrow ( u^1 \rightarrow u_{k_1})$ and let $u^2 =  u^1 \rightarrow u_{k_1}$.	Find $min(u^2) = a_{k_2}$. Then  $u^2 = u_{k_2} \rightarrow ( u^2 \rightarrow u_{k_2})$. In general, we can have $u^i = u_{k_i} \rightarrow ( u^i \rightarrow u_{k_i})$. Since $k_i < k_{i+1}$ and if $k_i > \frac{n}{2}$,then $u^i = \{a_{k_i}\}$ and halt since all $a_{k_i}$ where $k_i > \frac{n}{2}$ touches only locker $b_i$. This implies the last step of the algorithm becomes $ u^i \rightarrow u_{k_i} = \emptyset$ where $i$ is $\omega$ and $v = \{s_j\}^n $ where $s_j=\left\{\begin{array}{ccc}0 & \mathrm{if} & j \neq k_i \\ 1 & \mathrm{if} &  j = k_i \end{array}\right.$ by theorem 10.
\end{proof}

\begin{remark}
In the proof above, I used Theorem 2 as my basis. Alternatively, we can use set theory as our tool to clarify the above proof. In Remark 1, it is noted that the structure of $\left\langle S, \rightarrow \right\rangle$ is the same as the structure of $\left\langle \wp (N), \oplus  \right\rangle$, where $\wp(N)$ is the power set of $N$ and $\oplus$ is the symmetric difference operation. Since structure is preserved, then working on the abelian group $\left\langle \wp (N), \oplus  \right\rangle$ is same with working on $\left\langle S, \rightarrow \right\rangle$. Also, the "$u^2$" above means second iteration unlike the "$u^2$" in Theorem 2 which means $u \rightarrow u = \emptyset$.
\end{remark}

\begin{algorithm}
\textbf{\emph{Finding $u \in S$  given $v \in L$}}.\\
$u := \emptyset$\\
$m := \emptyset$\\
\textbf{for} $ j=1$ to $length(v)$\\
\{\textbf{if} $s_j = 1$,$m := u \rightarrow u_j$\\
 \textbf{else} $m$;\\
\}	
\end{algorithm}

The above algorithms seem simple but each part requires some algorithms too. Now, it may be noted that the definitional algorithm can be use instead. Yes, it can be used but the algorithms above comes natural with each other. Thus, the algorithms provided here are all-in-one package. Now let us have the following algorithms for each part.

\begin{algorithm}
\textbf{\emph{Finding Minimum  $a_k$ in $u \in S$}}.\\
\textbf{Input: }$a_j \in u$ \textbf{and }$|u| = n$\\ 
$ min(u) := a_{j_1}$\\
\textbf{for} $ i:= 2$ \textbf{to} $ n $\\
\textbf{if} $ min > a_{j_i}$ \textbf{then} $min := a_{j_i}$\\
\textbf{Output: } $min(u) $
\end{algorithm}

\begin{algorithm}
\textbf{\emph{Generating $u_k$}}.\\
\textbf{Input: }$\sigma(\{1\})$ \textbf{and }$|u| = n$\\ 
\textbf{for} $ i:= 1$ \textbf{to} $n$ \\
\textbf{while} $\sigma(\{1\})\cdot i \leq n$\\
 $\sigma(\{1\}) \cdot i \in u_i$ \\
\textbf{Output:} $u_1$ to $u_n$
\end{algorithm}

The other algorithms not stated above like \emph{set-symmetric difference algorithm} can be viewed at \emph{http://www.cplusplus.com/reference/algorithm/set\_symmetric\_difference/} and the algorithm generating for all element in $\sigma(\{1\})$ less than $n$ can be seen in the work of \cite{Atkin}. Now that we have the algorithms, there is a way, in a finite steps but may not be the most efficient algorithm, in solving the locker problem. Thus, concluding that the locker problem is now settled. This conclusion is made possible since the current algorithms and researchers I have seen about the problem are more of case to case basis not the one I made which encompasses all possible cases.

\end{document}